\documentclass[12pt, oneside]{amsart}       

\usepackage[T1]{fontenc}
\usepackage{palatino}

\usepackage{graphicx}

\usepackage{amsaddr}

                            \usepackage{amsthm}    
\newtheorem{theorem}{Theorem}[section]

\newtheorem{definition}{Definition}[section]
\newtheorem{lemma}[theorem]{Lemma}

\theoremstyle{remark}
\newtheorem*{remark}{Remark}
\usepackage{tcolorbox}

\title{Spectral Flow Equivariance for Calabi-Yau Sigma Models}
\author{Emile Bouaziz}
\address{Academia Sinica, Taipei}
\email{emile.g.bouaziz@gmail.com}

\begin{document}\maketitle

\begin{abstract} We write down an explicit operator on the chiral de Rham complex of a Calabi-Yau variety $X$ which intertwines the usual $\mathcal{N}=2$ module structure with its twist by the spectral flow automorphism of the $\mathcal{N}=2$, producing the expected \emph{spectral flow equivariance}. 
Taking the trace of the operators $L_{0}$ and $J_{0}$ on cohomology, and using the obvious interaction of spectral flow with characters, we obtain an explicit categorification of ellipticity of the elliptic genus of $X$, which is well known by other means. \end{abstract}

\section{introduction}\subsection{Sigma Models}

We are interested in mathematically constructing some expected additional symmetry in sigma models. Let $(X,\Omega^{+})$ be a $d$-dimensional Calabi-Yau variety, ie $X$ is a smooth $\mathbf{C}$ variety with a choice of volume form $\Omega^{+}$. Associated to $(X,\Omega^{+})$ is a sheaf $\Omega^{ch}_{X}$ of super vertex algebras on $X$, which is formally locally isomorphic to a tensor product of $d$-copies of the physicists' $bc\beta\gamma$-system. $\Omega^{ch}_{X}$ is referred to as the \emph{chiral de Rham complex} and was originally constructed in the seminal paper \cite{MSV} - it has since been extensively studied -  see for example \cite{KapVass}, \cite{Kap}, \cite{BHS} and in particular the recent work \cite{LinSong} in which global sections is calculated.

In \emph{loc. cit} the authors also construct an explicit action of the $\mathcal{N}=2$ superconformal algebra at central charge $c=3d$ on $\Omega^{ch}_{X}$ by showing that certain locally defined vectors are invariant under coordinate transformations preserving the volume form. In \cite{Kap} the relationship of $\Omega^{ch}_{X}$ is studied and it is explained that the cohomology $\mathcal{H}_{X}:=H^{*}(X,\Omega^{ch}_{X})$ is a rigorous mathematical incarnation of the state space of the physicists' $\sigma$-model with target $X$. 

Representations of the $\mathcal{N}=2$ are controlled by its Lie algebra of modes, $\mathfrak{g}_{\mathcal{N}=2}$. This Lie algebra has an automorphism referred to as the \emph{spectral flow} and denoted $\sigma$ and satisfying $$L(z)\mapsto L(z)+J(z)+\frac{d}{2}z^{-1},$$ $$J(z)\mapsto J(z)+dz^{-1},$$$$Q(z)\mapsto zQ(z),$$ $$G(z)\mapsto G(z)/z.$$ 

\subsection{What we do} In the physical literature it seems to be understood that $\mathcal{H}_{X}$ must admit an equivariant structure for the spectral flow, which we take to mean an isomorphism with its twist by $\sigma$. See for example \cite{Eg1}, \cite{Eg2} and in particular Odake's work \cite{Oda} from which it is quite clear that the results of this note are unsurprising to a physicist working with $\sigma$-models. 

Nonetheless, it is not immediately obvious how to produce the requisite isomorphism from the literature, and our conversations with some experts in the chiral de Rham theory have served to reassure that this is not \emph{entirely} a result of the author's ignorance. Our goal in this note is to remedy this situation, and make mathematically precise the extended $\mathcal{N}=2$ supersymmetry of $\Omega^{ch}_{X}$. More prosaically, we endow $\Omega^{ch}_{X}$ with an equivariant structure for $\sigma$ - writing $\sigma^{*}\Omega^{ch}_{X}$ for the same sheaf with $\mathcal{N}=2$ module structure twisted by $\sigma$ we shall show the following; 
\bigskip

\begin{tcolorbox}\begin{theorem}There is a natural isomorphism of sheaves of $\mathcal{N}=2$ modules on $X$,$$\sigma_{X}:\Omega^{ch}_{X}\rightarrow\sigma^{*}\Omega^{ch}_{X}.$$  Taking cohomology there is a natural isomorphism of $\mathfrak{g}_{\mathcal{N}=2}$ modules $$\mathcal{H}_{X}\rightarrow\sigma^{*}\mathcal{H}_{X}.$$ \end{theorem}\end{tcolorbox}

\subsection{Why we do this.} 

Some motivation for the results we present comes from forthcoming work of the author, where a version of $\Omega^{ch}_{X}$ is defined and studied with $X$ replaced with a log pair $(X;\Delta)$. We find that the amount of symmetry of $\Omega^{ch}_{X;\Delta}$ depends on how strong a CY condition we impose on the pair $(X;\Delta)$. We find that \begin{itemize} \item if $(X;\Delta)$ is \emph{weakly  CY}, ie if $X\setminus\Delta$ admits a volume form, then $\Omega^{ch}_{X;\Delta}$ admits $\mathcal{N}=2$ supersymmetry, \item and if further $(X;\Delta)$ is CY, ie $\Delta\sim -K_{X}$ is anticanonical, then $\Omega^{ch}_{X;\Delta}$ is also equivariant for the spectral flow. \end{itemize}

We sketch some further motivation coming from the elliptic genus. In fact the result of this note can be seen as a categorification of a functional equation satisfied by the elliptic genus. More precisely, it is proven in \cite{BorLib} that for $X$ proper, the $\mathcal{N}=2$ character of $\mathcal{H}_{X}$, defined as $$\operatorname{Trace}\Big\{q^{L_{0}-c/24}y^{J_{0}}(-1)^{F_{0}}:\mathcal{H}_{X}\longrightarrow\mathcal{H}_{X}\Big\},$$  is a \emph{Jacobi form} of index $d/2$ and weight $0$. In  particular it satisfies the elliptic equation $$q^{dn^{2}/2}y^{dn}\operatorname{ch}_{\mathcal{N}=2}(\mathcal{H}_{X})(q,q^{n}y)=\operatorname{ch}_{\mathcal{N}=2}(\mathcal{H}_{X})(q,y).$$  This is proven via an explicit calculation with $\theta$-functions, using a degeneration of $\Omega^{ch}_{X}$ to a product of various tensor fields. An explicit examination of the form of the spectral flow makes clear that this follows immediately from our main theorem, indeed we have  $$\operatorname{ch}_{\mathcal{N}=2}(q,y)((\sigma^{n})^{*}V)=q^{dn^{2}/2}y^{dn}\operatorname{ch}_{\mathcal{N}=2}(V)(q,y)$$ for any $\mathcal{N}=2$ module $V$ at central charge $c=3d$ with well defined character. As such we can think of our main result as an explicit lifting to the level of vector spaces (in fact sheaves) of ellipticity of the character.

\subsection{What makes this work} Not all $\mathcal{N}=2$ modules admit spectral flow equivariance, indeed the elliptic transformations of their characters form an obvious obstruction. So something special about $\mathcal{H}_{X}$ makes our arguments work. We use this subsection to explain this, notation will be explained in the following sections, and the reader unfamiliar with the subject matter should read them first.

A formula for equivariance is given below in \ref{maindef}, where relevant definitions are spelled out. Writing $E^{\pm}_{v}(z)$ for the strictly positive and strictly negative parts of the exponential of a field $v(z)$, and $\Omega^{+}(z)$ for the field corresponding to the volume form $\Omega^{+}$, it is given by $$\frac{1}{2\pi i}\oint(-1)^{F_{0}}E^{+}_{-J}(z)\Omega^{+}(z)E^{-}_{-J}(z)z^{-1-J_{0}}dz.$$ \begin{remark}In fact the integrand is a constant multiple of $\frac{dz}{z}$, and we have included the contour integral above is just so that the result is manifestly an endomorphism of $\mathcal{H}_{X}$ and not a general field on it.\end{remark} We will eventually see that for this definition to do the trick what we need are two vectors $\Omega^{\pm}$, which we will take to be the volume form and its inverse, such that the following list of properties is satisfied.\begin{enumerate} \item $\Omega^{\pm}$ satisfy the OPEs with the even current given by $J$; $$J(z)\Omega^{\pm}(w)\sim\frac{\pm d\Omega^{\pm}(w)}{z-w}.$$ \item With respect to the odd currents $Q$ and $G$ we have regular OPEs $$Q(z)\Omega^{+}(w)\sim 0$$ $$G(z)\Omega^{-}(w)\sim 0.$$ \item We have an OPE $$\Omega^{+}(z)\Omega^{-}(w)\sim\frac{1}{(z-w)^{d}}+\Big(\operatorname{less}\,\operatorname{singular}\,\operatorname{at} z=w\Big).$$ \item We have normally ordered products $$:J(z)\Omega^{\pm}(z):=\pm\frac{d}{dz}\Omega^{\pm}(z).$$ \end{enumerate}

\begin{remark} In \cite{Oda} Odake considers the above OPEs, and whilst it is not explained in \emph{loc. cit.} precisely how to get the spectral flow equivariance for $\mathcal{H}_{X}$ from these, it seems pretty clear that the author is aware that this can be done.\end{remark}

\subsection{Acknowledgements} We've benefited from conversations with numerous mathematicians and physicists. Thanks to Mart\'{i} Rosell\'{o}, John Duncan, Reimundo Heluani, Gurbir Dhillon and in particular Andrew Linshaw for helpful correspondence and suggestions.

\section{Brief recollections on the set-up} \subsection{Vertex language} We will provide very little background on the objects in question, we just fix some notation. Vertex algebras are always assumed to be super such. If $V$ is a vertex algebra containing a vector $v$ then we write $$v(z)=\sum_{i}v_{i}z^{-1-i}$$ for the \emph{field} generated by $v$. We write $\operatorname{vac}$ for the vacuum vector and $\partial$ for the translation operator. We abbreviate $v_{(-1)}w$ as $vw$ throughout. We write $:\phi(z)\psi(z):$ for the normally ordered product of fields $\phi$ and $\psi$. We encode commutators between the modes of fields as OPEs as is standard, cf. \cite{Kac}

Given a vector $\alpha$ we define the \emph{vertex operators} $E^{\pm}_{\alpha}(z)$ by  $$E^{+}_{\alpha}(z):=\operatorname{exp}\Big(\sum_{n=1}^{\infty}\frac{\alpha_{-n}z^{n}}{n}\Big),$$ $$E^{-}_{\alpha}(z):=\operatorname{exp}\Big(\sum_{n=1}^{\infty}\frac{\alpha_{n}z^{-n}}{-n}\Big).$$ Assuming also that $\alpha_{0}$ acts semisimply on $V$ with integral eigenvalues, we can define the field $z^{\alpha_{0}}$. If $\alpha$ and $\beta$ are two vectors in $V$ satisfying the OPE $$\alpha(z)\beta(w)\sim\frac{N}{(z-w)^{2}}$$ for an integer $N$, then we have, \emph{cf} \cite{Kac}, $$E^{+}_{\alpha}(z)E^{-}_{\beta}(w)=\big(1-z/w\big)^{N}E^{-}_{\beta}(w)E^{+}_{\alpha}(z).$$ 

\subsection{The chiral de Rham as an $\mathcal{N}=2$ module} In \cite{MSV} the authors show how to the tensor product of $d$ copies of  the $bc\beta\gamma$ system, $V_{bc\beta\gamma}$, to a sheaf of vertex algebras on $X$. More precisely they work with a certain completion with respect to the vectors $\gamma^{i}_{0}\operatorname{vac}$, \emph{cf} subsection 3.1. of \emph{loc. cit.} We will denote the resulting completion also as $V_{bc\beta\gamma}$ by a slight abuse of notation. Formally locally on $X$, the result is thus isomorphic as a vertex algebra to  $V_{bc\beta\gamma}^{\otimes d}$, which we recall is generated by bosonic fields $\{\gamma^{i}(z),\beta_{i}(z)\}$ and fermionic fields $\{b_{i}(z), c^{i}(z)\}$, for $1=1,...,d$, subject only to the OPEs $$\gamma^{i}(z)\beta_{j}(w)\sim\frac{\delta_{ij}}{(z-w)}$$ $$b_{i}(z)c^{j}(w)\sim\frac{\delta_{ij}}{(z-w)}.$$ \begin{remark} With respect to these coordinates, the volume form $\Omega^{+}$ is locally expressed as $\Omega^{+}=c^{1}...c^{d}$. Denoting by $\Omega^{-}$ the polyvector field dual to $\Omega^{+}$, $\Omega^{-}$ is expressed as $b_{1}...b_{d}$. \end{remark}

The $\mathcal{N}=2$ superconformal vertex algebra, here denoted $V_{\mathcal{N}=2}$ is generated by fields denoted traditionally as $L,J,Q,G$, subject to various OPEs that the reader can find in \cite{MSV} subsection 2.1. We denote by $\mathfrak{g}_{\mathcal{N}=2}$ the Lie algebra of modes of $V_{\mathcal{N}=2}$. We will make use only of the following OPEs in this note; $$J(z)J(w)\sim\frac{d}{(z-w)^{2}},$$ $$J(z)Q(w)\sim\frac{Q(w)}{z-w},$$ $$J(z)G(w)\sim \frac{-G(w)}{(z-w)}.$$ Then theorem 4.2 of \cite{MSV} constructs a morphism $V_{\mathcal{N}=2}\rightarrow\Omega^{ch}_{X}$, or equivalently global sections $L,J,Q,G$ of $\Omega^{ch}_{X}$ satisfying the OPEs of the identically named vectors in $V_{\mathcal{N}=2}$.
\begin{remark} In \cite{MSV} the authors work with the so-called topological vertex algebra at rank $d$. This is obtained as a \emph{topological twist} of $V_{\mathcal{N}=2}$, and in particular the underlying vertex algebras are the same, they differ only in choice of a Virasoro via the \emph{twist} $$L\mapsto L\pm\frac{1}{2}\partial J.$$ \end{remark} 

For reference, we recall the local form of the global sections $L,J,G,Q$ in $V_{bc\beta\gamma}^{\otimes d}$. We will only use explicitly the local form of the global section $J$, which we recall is given as $\sum_{i}c^{i}b_{i}$.

\begin{definition} The spectral flow automorphism, $\sigma$, of the Lie algebra $\mathfrak{g}_{\mathcal{N}=2}$ at central charge $c=3d$  is defined by the following formulae, which we choose to express in terms of fields; \begin{itemize}\item $L(z)\mapsto L(z)+J(z)+\frac{d}{2}z^{-1}$, \item $J(z)\mapsto J(z)+dz^{-1}$, \item $Q(z)\mapsto zQ(z)$, \item $G(z)\mapsto G(z)/z.$ \end{itemize}\end{definition}

\begin{definition} We denote by $\sigma^{*}\Omega^{ch}_{X}$ the sheaf of modules for $\mathfrak{g}_{\mathcal{N}=2}$ defined by twisting by $\sigma$. Namely, if $x\in\mathfrak{g}_{\mathcal{N}=2}$ and $s$ is a local section of $\Omega^{ch}_{X}$, the twisted action is given by $x.s:=\sigma(x)s$. \end{definition}

\section{Constructing equivariance} We will show in this section how to construct the desired isomorphism $\sigma_{X}$ between $\Omega^{ch}$ and its spectral flow twist. First we write down a certain field on $\Omega^{ch}_{X}$. Then we check that it is actually a constant field, and can be considered as an endomorphism. We then check that it has intertwines the untwisted and twisted $\mathcal{N}=2$ module structures. Finally, we construct an inverse for it.

\subsection{A formula for $\sigma_{X}$} Let $J$ be the global section of $\Omega^{ch}_{X}$ defined by the vector $J$, and recall that we have the OPE $$J(z)J(w)\sim\frac{D}{(z-w)^{2}}.$$ Further recall the fields $E^{\pm}_{J}(z)$ introduced above and the operator $(-1)^{F_{0}}$ which is $1$ on Bosons and $-1$ on Fermions.

\label{maindef}\begin{definition} Let $X$ be a Calabi-Yau with volume form $\Omega^{+}$. Define the field $$\sigma_{X}(z):=(-1)^{F_{0}}E^{+}_{-J}(z)\Omega^{+}(z)E^{-}_{-J}(z)z^{-J_{0}}.$$ \end{definition}

This turns out to be independent of $z$, as we shall see in the lemma below. We will denote by $\sigma_{X}$ the resulting endomorphism of $\Omega^{ch}_{X}$.

\label{constant}\begin{lemma}The field $\sigma_{X}(z)$ is constant, that is to say we have $\frac{d}{dz}(\sigma(z))=0.$\end{lemma}

\begin{proof} Differentiating the definition of $\sigma(z)$, this comes down to the equation for the normally ordered product $:J(z)\Omega^{+}(z):=\frac{d}{dz}\Omega^{+}(z).$ This is equivalent to $J\Omega^{+}=\partial\Omega^{+}.$, which follows from the expressions in local coordinates; $$J=\sum_{i}c^{i}b_{i},$$ $$\Omega^{+}=c^{1}c^{2}...c^{d}.$$ Indeed, expanding out the $(-1)$ mode of the expression for $J$ using Borcherds' formula we see that only the summands $c^{i}_{(-2)}b_{i,(0)}$ act non-trivially. That they act as desired is clear from the OPEs $$b_{j}(z)c^{i}(w)\sim\frac{\delta_{ij}}{z-w},$$ and we are done.\end{proof}

\subsection{Constructing the inverse} Recall we let $\Omega^{-}$ be the polyvector on $X$ inverse to the volume form $\Omega^{+}$. We define now the field $$\tau_{X}(w):=(-1)^{F_{0}}E^{+}_{J}(w)\Omega^{-}(w)E^{-}_{J}(w)w^{J_{0}}.$$ As in the case of \ref{constant} we have the following lemma which allows us to treat the field $\tau_{X}(w)$ as a constant $\tau_{X}$.

\begin{lemma} We have $\frac{d}{dw}\tau_{X}(w)=0$. \end{lemma}
\begin{proof} Exactly as in \ref{constant}, except we now have $J\Omega^{-}=-\partial\Omega^{-}$. \end{proof}

Now we prove that $\tau_{X}$ is the inverse to $\sigma_{X}$.

\begin{proof} By definition we must show that we have an equality of fields $$E^{+}_{-J}(z)\Omega^{+}(z)E^{-}_{-J}(z)z^{-J_{0}}E^{+}_{J}(w)\Omega^{-}(w)E^{-}_{J}(w)w^{J_{0}}=\operatorname{id}.$$ Our strategy is to rewrite the right hand side in a manner that makes the $z=w$ specialisation, which we note must exist as the field in question is actually constant, more readily computable. In order to do this we compare the above expression to the field $$E^{+}_{-J}(z)E^{+}_{J}(w)\Omega^{+}(z)\Omega^{-}(w)E^{-}_{-J}(z)E^{-}_{J}(w)z^{-J_{0}}w^{J_{0}}.$$ We must now account for the various left to right switches in the order of the individual terms in the product. \begin{enumerate}\item Swapping $z^{-J_{0}}$ and $\Omega^{-}(w)$ introduces a factor $z^{d}$. \item Swapping $E^{-}_{-J}(z)$ and $E^{+}_{J}(w)$ introduces a factor $\big(1-w/z\big)^{-d}$.\item Swapping $E^{+}_{J}(w)$ and $\Omega^{+}(z)$ introduces a factor  $\big(1-w/z\big)^{d}$. \item Swapping $E^{-}_{-J}(z)$ and $\Omega^{-}(w)$ introduces a factor  $\big(1-w/z\big)^{d}$.\end{enumerate} The first point simply states that $\Omega^{-}$ is of weight $-d$ with respect to $J_{0}$. The second encodes the OPE $$J(z)J(w)\sim\frac{d}{(z-w)^{2}}.$$ The third and fourth follow respectively from the OPEs $$J(z)\Omega^{+}(w)\sim \frac{d\Omega^{+}(w)}{z-w},$$ $$J(z)\Omega^{-}(w)\sim \frac{-d\Omega^{-}(w)}{z-w},$$ which in term follow from the fact that $J_{>0}$ annihilate $\Omega^{+}$ and $\Omega^{-}$, as they are minimal conformal weight states with $J_{0}$ eigenvalues $d$ and $-d$ respectively. Combining the above, the contributions from the second and third terms cancel and we obtain a total factor of $$z^{d}\Big(1-w/z\Big)^{d}=(z-w)^{d}.$$ We find thus $$\sigma_{X}(z)\tau_{X}(w)=(z-w)^{d}E^{+}_{-J}(z)E^{+}_{J}(w)\Omega^{+}(z)\Omega^{-}(w)E^{-}_{-J}(z)E^{-}_{J}(w)z^{-J_{0}}w^{J_{0}}.$$ Now the expressions in local coordinates, $\Omega^{+}=c^{1}...c^{d}$ and $\Omega^{-}=b_{1}...b_{d}$, imply an OPE $$\Omega^{+}(z)\Omega^{-}(w)\sim\frac{1}{(z-w)^{d}}+\Big(\operatorname{less}\,\operatorname{singular}\,\operatorname{at} z=w\Big).$$ The coefficient of $(z-w)^{d}$ then ensures that the product is regular at $z=w$. Now we simply set $z=w$. The exponential terms will not contribute as $E^{\pm}_{-J}(z)E^{\pm}_{J}(w)|_{z=w}=1$. Then the only contribution comes from the most singular term of the OPE for $\Omega^{+}$ and $\Omega^{-}$, which contributes the identity as desired. This concludes the proof.\end{proof}

\subsection{Intertwining the twisted and untwisted actions} We now examine how $\sigma_{X}$ interacts with the action of the $\mathcal{N}=2$.

\begin{lemma} $\sigma_{X}$ is a morphism of $\mathcal{N}=2$ modules from $\Omega^{ch}_{X}\longrightarrow\, \sigma^{*}\Omega^{ch}_{X}.$ \end{lemma}

\begin{proof} It is standard, and easy to check, that $Q(w)$ and $G(w)$ generate the $\mathcal{N}=2$, so we need only check compatibility of $\sigma_{X}$ with these. Then it suffices to check that we have the following two relations;

$$\sigma_{X}Q(w)=wQ(w)\sigma_{X},$$ $$\tau_{X}G(w)=wG(w)\tau_{X},$$ where we note that there is an evident symmetry between the two formulae. We will show the first, leaving the near identical proof of the second for the reader.

\begin{itemize}\item We first note the following trivial relations, we have $$z^{J_{0}}Q(w)z^{-J_{0}}=zQ(w),\,(-1)^{F_{0}}Q(w)(-1)^{F_{0}}=-Q(w),$$ as $Q$ is Fermionic and has weight one with respect to $J_{0}$. \item Now we recall the very general identity $e^{x}ye^{-x}=e^{\operatorname{Adj}(x)}(y)$, valid in any topological Lie algebra containing elements $x$ and $y$ such that $\operatorname{Adj}(x)^{n}(y)$ converges as $n\rightarrow\infty$. Using this we compute $$E^{+}_{J}(z)Q(w)E^{+}_{-J}(z)=\operatorname{exp}\Big(\operatorname{Adj}\Big(\sum_{n=1}^{\infty}\frac{J_{-n}z^{n}}{n}\Big)\Big)Q(w).$$ \item Recalling the OPE $J(z)Q(w)\sim \frac{Q(w)}{z-w},$ we compute that $$E^{+}_{J}(z)Q(w)E^{+}_{-J}(z)=\big(1-z/w\,\big)^{-1}Q(w).$$ An identical computation implies we have $$E^{+}_{J}(z)Q(w)E^{+}_{-J}(z)=\big(1-w/z\,\big)Q(w).$$ \item $Q(w)$ commutes with the field $\Omega^{+}(z)$.\end{itemize} Putting this all together and recalling that we have $$\sigma_{X}(z)=(-1)^{F_{0}}E^{+}_{-J}(z)\Omega^{+}(z)E^{-}_{-J}(z)z^{-J_{0}},$$ we find a total factor of $$(-1)z\big(1-z/w\,\big)^{-1}\big(1-w/z\,\big)=w,$$ upon moving $Q(w)$ from left to right. This proves the claim.\end{proof}

Putting the above together, we arrive at our main theorem, which we restate below.

\begin{theorem}There is a natural isomorphism, $\sigma_{X}:\Omega^{ch}_{X}\rightarrow\sigma^{*}\Omega^{ch}_{X},$  of sheaves of modules for the Lie algebra $\mathfrak{g}_{\mathcal{N}=2}$. Taking cohomology there is a natural isomorphism of $\mathfrak{g}_{\mathcal{N}=2}$ modules, $\mathcal{H}_{X}\rightarrow\sigma^{*}\mathcal{H}_{X}$. \end{theorem}


\begin{thebibliography}{9}

\bibitem{BHS}

D. Ben-Zvi, R. Heluani, M. Szczesny.
\textit{ Supersymmetry of the chiral de Rham complex}. 
Compositio Mathematica. 2008;144(2):503-521.

\bibitem{BorLib}
L. Borisov, A. Libgober,
\textit{Elliptic genera of toric varieties and applications to mirror symmetry,}
Invent. Math. 140 pp. 453-485. 

\bibitem{Eg1}
T. Eguchi,
\textit{Compact Formulas for the Completed Mock Modular Forms,}
JHEP 2014

\bibitem{Eg2}
T. Eguchi,
\textit{Modular bootstrap of boundary N=2 Liouville theory,}
Comptes Rendus Physique 6 (2005) 209-217

\bibitem{Kac}
V. Kac,
\textit{Vertex Algebras for Beginners,}
University Lecture Series, AMS (1998)

\bibitem{Kap}
A. Kapustin,
\textit{Chiral de Rham complex and half twisted sigma model,} 
arXiv:hep-th/0504074

\bibitem{KapVass}
M. Kapranov, E. Vasserot,
\textit{Vertex algebras and the formal loop space,}
Pub. Math. de l'IHES, vol. 100 (2004), pp. 209-269.

\bibitem{LinSong}
A. Linshaw, B. Song,
\textit{The Global Sections of Chiral de Rham Complexes on Compact Ricci-flat Kahler Manifolds II,}
Commun. Math. Phys. 399, 189–202 (2023).

\bibitem{MSV}
F. Malikov, V. Schechtman, A. Vaintrob,
\textit{Chiral de Rham Complex,}
Commun. Math. Phys. 204, 439-473

\bibitem{Oda}
S. Odake,
\textit{c=3d conformal algebra with extended supersymmetry,}
Modern Physics Letters A 1990 05:08, 561-580
















\end{thebibliography}
\end{document}